\documentclass[a4paper,11pt]{article}
\title{Relative model completeness of henselian valued fields with finite ramification and various value groups}
\author{Anna De Mase\protect{\footnote{Dipartimento di Matematica e Fisica, Università degli Studi della Campania ``Luigi Vanvitelli", viale Lincoln 5, Caserta (Italy), anna.demase@unicampania.it; ORCID: 0000-0003-4902-6001 \\ Partially supported by PRIN 2022 - "Models, sets and classifications".}}}
\date{}
\usepackage{amsmath}
\usepackage{amsthm}
\usepackage{wasysym}
\usepackage{amssymb}
\usepackage{amsfonts}
\usepackage{tikz}
\usepackage[hyperfootnotes]{hyperref}
\usepackage{times}
\usetikzlibrary{matrix}
\begin{document}
\maketitle
\newtheorem*{thm*}{Theorem}
\newtheorem{thm} {Theorem}[section]
\newtheorem{defn} [thm] {Definition}
\newtheorem{cor} [thm] {Corollary}
\newtheorem{rem} [thm] {Remark}
\newtheorem{lem} [thm] {Lemma}
\newtheorem{prop} [thm] {Proposition}
\newtheorem{fact} [thm] {Fact}
\begin{abstract}
We investigate the model completeness of the theory of a mixed characteristic henselian valued field with finite ramification relative to the residue field and value group. We address the case in which the valued field has a value group with finite spines, and the case in which the value group is elementarily equivalent to the infinite lexicographic sum of $\mathbb{Z}$ with a minimal positive element. In both cases, we find a one-sorted language in which the theory of the valued field is model complete, if the theory of the residue field is model complete in the language of rings.
\end{abstract}
\vspace{2mm}
\section*{Introduction}
Object of interest in research in the model theory of valued fields is identifying the model theoretical properties that can be transferred from the value group and the residue field to the valued field. The Ax-Kochen/Ershov (AKE) principle (Theorem \ref{AKE}) obtained by Ax and Kochen in \cite{AxK65i} and independently by Ershov in \cite{Ershov65} has been the first transfer principle for henselian valued fields, stating that two equicharacteristic zero henselian valued fields are elementary equivalent in the three-sorted language $\mathfrak{L}_{vf}$ of valued fields if and only if their residue fields and value groups are elementarily equivalent in the language $\mathfrak{L}_{ring}$ of rings and the language $\mathfrak{L}_{oag}$ of ordered abelian groups, respectively. The same result was proved for mixed characteristic $(0,p)$ henselian valued fields, by Belair (\cite[Corollary 5.2]{BEL99}) in the case of unramified valuations (when the valuation of $p$ is minimal positive in the value group) and perfect residue fields, and recently by Anscombe and Jahnke (\cite{AJ21}) in the unramified mixed characteristic case and arbitrary residue field. Some AKE-like principles have been given also for expansions of henselian valued fields. In \cite{ACGZ22} M. Aschenbrenner, A. Chernikov, A. Gehret, and M. Ziegler provide an Ax-Kochen/Ershov-like characterization for distal expansions of finitely ramified henselian valued fields, where distality is a combinatorial property of the theory. In the cases of pure unramified valued fields, if $K_{1}$ and $K_{2}$ are two henselian valued fields such that $K_{1}\subseteq K_{2}$, then the embedding of $K_{1}$ into $K_{2}$ is elementary if and only if the embeddings of the residue fields and the value groups are elementary. This implies that the theory of a henselian valued field $K$, either of equicharacteristic zero or unramified of mixed characteristic, is model complete relative to the theories of the residue field and the value group. Notice that an AKE-like theorem does not hold in the case of two mixed characteristic valued fields with finite ramification (i.e. when the valuation of $p$ is strictly greater than the minimal positive element of the value group) not even in the case in which both valued fields have the same ramification. Indeed, the theory of the residue field in $\mathfrak{L}_{ring}$, the theory of the value group in $\mathfrak{L}_{oag}$, and a fixed ramification do not determine the theory of the valued field in $\mathfrak{L}_{vf}$ up to elementary equivalence (see \cite[Example 2.3]{ADJ23}). In the latter paper (\cite{ADJ23}), in particular, the authors show that an AKE-like theorem holds at the cost of considering further structure on the residue field. However, assuming the model completeness of the residue field in the language of rings, results on the relative model completeness of the theory of a fixed valued field hold. Indeed, in \cite{MD16} Derakshan and Macintyre prove model completeness of the theory of a henselian valued field with finite ramification and valued in a $\mathbb{Z}$-group relative to the residue field. In particular, they show the following
\begin{thm*}[{ \cite[Theorem 1]{MD16}}]
Let $K$ be a henselian valued field of mixed characteristic with finite ramification. Suppose the value group of $K$ is a $\mathbb{Z}$-group. If the theory of the residue field of $K$ is model complete in the language of rings, then the theory of $K$ is model complete in the language of rings.
\end{thm*}
Here we aim at generalizing their result to henselian mixed characteristic valued fields with finite ramification and value groups which are not $\mathbb{Z}$-groups. In particular, we stress that these generalizations give relative model completeness of the valued fields in one-sorted languages that are expansions of the language of rings. $\\$
In the first section, we consider the case of a value group with finite spines, i.e. with finitely or countably many definable convex subgroups (Theorem \ref{TheoremModelCompletenessFiniteRamification}). For this result, an analysis of the model theory of ordered abelian groups has been necessary, referring in particular to \cite{CH11} and \cite{HH18}. An analogous result has been obtained in the case of a value group elementarily equivalent to the infinite lexicographic sum of $\mathbb{Z}$ with a minimal positive element (Theorem \ref{TheoremModelCompletenessProductZ}). In the last section, we provide examples of henselian valued fields with the previous value groups, and these fields are obtained as fields of power series with one, finitely many, or countably many variables with their respective valuations.


\section{Notations}
In the model-theoretic analysis of a valued field, different languages have been used according to which algebraic aspect one wants to investigate. The languages that are usually considered are either one-sorted or multisorted. We commonly use the following notations:
\begin{itemize} \item $\Gamma$ for the value group;
\item $O$ for the valuation ring;
\item $M$ for the maximal ideal;
\item $k$ for the residue field;
\item $(K,v)$, $(K,O)$, $(K,\Gamma,k)$, $(K,k,\Gamma,v,res)$ (where $v:K\longrightarrow\Gamma$ is the valuation and $res:K\longrightarrow k$ the residue map), will denote the structure of a valued field according to the different properties we intend to analyze.
\end{itemize}
We may use a subscript to specify the valuation or the field we are working with.
The language of valued fields is the three sorted language $$\mathfrak{L}_{vf}=(\mathfrak{L}_{Ring},\mathfrak{L}_{oag},\mathfrak{L}_{ring},v,res),$$ where $\mathfrak{L}_{Ring}=\{+,\cdot,0,1\}$ is the language for the valued field sort, $\mathfrak{L}_{oag}=\{+,0,\leq\}$ for the value group sort, $\mathfrak{L}_{ring}=\{+,\cdot,0,1\}$ for the residue field sort, and $v,res$ are two function symbols interpreted respectively as the valuation $v:K\longrightarrow\Gamma$ and the residue map $res:K\longrightarrow k$, where $res(x)=0$ for all $x\notin O_{K}$. 
Throughout the paper, if $\mathfrak{L}$ is a language and $\mathcal{M}$ is an $\mathfrak{L}$-structure, with abuse of notation, we will usually identify a predicate $A\in\mathfrak{L}$ with its interpretation $A^{\mathcal{M}}$ in $\mathcal{M}$.
We say that a valued field $(K,k,\Gamma,v,res)$ is of equicharacteristic $0$ (or characteristic $(0,0)\,$) if $char(K)=char(k)=0$. We say that it is of mixed characteristic (or characteristic $(0,p)\,$) if $char(K)=0$ and $char(k)=p$. The equicharacteristic $p$ case ($(p,p)$) will not be addressed in this paper.
We say that a valued field $(K,v)$ is discrete if the order defined on its value group $G$ is discrete, i.e. $G$ has a minimal positive element. 
We will often use the abbreviation "oag" in referring to an ordered abelian group.

Let $G, G'$ be two arbitrary groups. We write $G\subseteq_{fg}G'$ if $G$ is a finitely generated subgroup of $G'$. 

\section{Ordered abelian groups with finite spines} \label{OAGfs}
The model theory of ordered abelian groups has been studied for the last 50 years. Significant results have been achieved on (relative) quantifier elimination, first by Schmitt in \cite{Schmitt82}, and more recently by Cluckers and Halupczok in \cite{CH11} with a different though equivalent language. The case of oags with finite spines, or equivalently oags with finite $n$-regular rank, has been considered by Halevi and Hasson in \cite{HH18}, by Dolich and Goodrick in \cite{DG18} and by Farré in \cite{F17}. In this section, we recall the definition of oags with finite spines and give a language for model completeness. \\
\begin{defn} Let $G$ be an ordered abelian group. For each $n\in\mathbb{N}$, we recall the following notions:
\begin{itemize}
    \item for $a\in G\setminus nG$, let $H_{a}$ be the largest convex subgroup of $G$ such that $a\notin H_{a}+nG$; set $H_{a}=0$ if $a\in nG$. Define $S_{n}:=G/{\sim}$, where $a\sim a^{'}$ iff $H_{a}=H_{a^{'}}$ and let $s_{n}:G\longrightarrow S_{n}$ be the canonical projection. For $\alpha=s_{n}(a)\in S_{n}$, define $\overline{H}_{\alpha}:=H_{a}$;\\
    \item for $b\in G$, set $H_{b}^{'}:=\bigcup_{\alpha\in S_{n}, b\notin \overline{H}_{\alpha}}\overline{H}_{\alpha}$, where the union over the empty set is 0. 
    Define $T_{n}:=G/{\dot{\sim}}$, where $b\,\,\dot{\sim}\,\,b^{'}$ iff $H'_{b}=H'_{b^{'}}$ and let $t_{n}:G\longrightarrow T_{n}$ be the canonical projection. For $\alpha=t_{n}(b)\in T_{n}$, define $\overline{H}_{\alpha}:=H_{b}^{'}$;\\
    \item for $\beta\in T_{n}$, set $\overline{H}_{\beta^{+}}=\bigcap_{\substack{\alpha\in S_{n} \\ \overline{H}_{\beta}\subset \overline{H}_{\alpha}}}\overline{H}_{\alpha}$, where the intersection over the empty set is $G$. Here, $\beta^{+}$ is viewed as an element of a copy of $T_{n}$, denoted by $T_{n}^{+}$;
    \item define a total preorder on $\bigcup_{n\in\mathbb{N}}S_{n}\cup T_{n}\cup T_{n}^{+}$ by $\alpha\leq\alpha^{'}$ iff $\overline{H}_{\alpha}\subseteq \overline{H}_{\alpha^{'}}$.
\end{itemize}
For every $n\in\mathbb{N}$, the structure $(S_{n}\cup T_{n} \cup T_{n}^{+},\leq)$ is the \emph{$n$-spine} of $G$.
\end{defn}

We say that $G$ has \textit{finite spines} if $|S_{n}|$ is finite for each $n\in\mathbb{N}$.
\begin{rem} If $G$ has finite spines, then for every $n\in\mathbb{N}$ the sorts $T_{n}$, $T_{n}^{+}$ do not add any contribution in terms of new convex subgroups.\end{rem}
By \cite[Lemma 2.1]{CH11}, all the $\overline{H}_{\alpha}$ are definable in $\mathfrak{L}_{oag}$ and the ordered structure $(S_{n},<)$ is interpretable in $G$ for every $n\in\mathbb{N}$. Moreover, if $G$ is a group with finite spines, then $$\left\{\overline{H}_{\alpha}\vert\,\alpha\in S_{n},\,n\in\mathbb{N}\right\}$$ are all the definable convex subgroups of $G$ \cite[Proposition 3.3]{HH18}. In particular, $G$ has only finitely or countably many definable convex subgroups. From now on, we fix an enumeration $(H_{i})_{i\in I}$, $I\subseteq\mathbb{N}$, of all the definable convex subgroups of $G$. By convention $H_{0}=\{0\}$.

Oags with finite spines admit quantifier elimination in a language that is simpler than the language in which arbitrary oags admit elimination of quantifiers. In particular, we recall the following proposition from \cite{HH18} obtained by a reduction of Cluckers-Halupczok language (\cite{CH11}) to the case of groups with finite spines.

\begin{prop}[{ \cite[Proposition 3.4]{HH18}}] \label{QEFS} Let $G$ be an ordered abelian group with finite spines and let $\{H_{i}\}_{i\in I}$ be an enumeration of the definable convex subgroups of $G$ for some finite or countable set of indeces $I$. Then the definable expansion of $G$ $$(G,0,+,-,\leq, x=_{H_{i}}y+j1_{i}, \,x\equiv_{m,H_{i}}y+j1_{i})_{j\in\mathbb{Z},\,i\in I,\,m\in\mathbb{N}},$$admits quantifier elimination, where \begin{itemize} \item $j1_{i}$ is $j$ times a representative $1_{i}\in G$ of the minimal positive element of the quotient $G/H_{i}$ if it exists, $0$ otherwise; \item for all $x,y\in G, x=_{H_{i}}y+j1_{i}$ if and only if $x-y-j1_{i}\in H_{i}$; \item for all $x,y\in G, x\equiv_{m,H_{i}}y+j1_{i}$ if and only if there exists $t\in G$ such that $x-y-mt-j1_{i}\in H_{i}$.\end{itemize}\end{prop}

We are now in a position to identify a language in which the theory of an oag with finite spines $G$ is model complete.
\begin{prop} \label{PropMCFS} Let $G$ be an ordered abelian group with finite spines and let $\{H_{i}\}_{i\in I}$ be an enumeration of the definable convex subgroups. Then $Th(G)$ is model complete in the language $\mathfrak{L}_{oag}^{*}=\{0,+,-,\leq,(j1_{i}+H_{i})_{i\in I,j=0,1}\}$, where $j1_{i}$ is $j$ times a representative $1_{i}\in G$ of the minimal positive element of the quotient $G/H_{i}$ if it exists, $0$ otherwise.\end{prop}
\begin{proof}  From Proposition \ref{QEFS}, we replace the relations $x=_{H_{i}}y+j1_{i}$ and $x\equiv_{m,H_{i}}y+j1_{i}$ by the predicates $j1_{i}+H_{i}$. Note that, for $j=0$, $H_{i}$ is a predicate of the language, and for $j=1$, $1_{i}+H_{i}$ is a predicate for the minimal positive element of the quotient if it exists. So, for any $x,y\in G$ and $j\in\mathbb{Z}$, we have $x=_{H_{i}}y+j1_{i}$ if and only if $x-y\in j1_{i}+H_{i}$, that is \begin{align*}x=_{H_{i}}y+j1_{i} &\Longleftrightarrow \exists z\,\,(z=x-y\,\wedge z\in j1_{i}+H_{i})\end{align*} and \begin{align*} x=_{H_{i}}y+j1_{i}&\Longleftrightarrow \forall z\,\,\left(z\in j1_{i}+H_{i}\rightarrow x-y-z\in H_{i}\right).\end{align*} Moreover, we have $x\equiv_{m,H_{i}}y+j1_{i}$ if and only if  $x-y-mt\in j1_{i}+H_{i}$ for some $t\in G$, that is \begin{align*} x\equiv_{m,H_{i}}y+j1_{i}&\Longleftrightarrow \exists z,t\,\,(z=x-y-mt\,\wedge z\in j1_{i}+H_{i})\end{align*} and \begin{align*} x\not\equiv_{m,H_{i}}y+j1_{i}&\Longleftrightarrow \exists t\,\,(x-y-mt\in H_{i}\,\vee \\ &\,\,\,\,\,\,\,\vee\,x-y-mt\in (1+j)1_{i}+H_{i}\,\vee\, \\ &\,\,\,\,\,\,\,\vee x-y-mt\in(2+j)1_{i}+H_{i}\,\vee\ldots \\ &\ldots\vee\,x-y-mt\in(m-1+j)1_{i}+H_{i}).\end{align*} So the relations are existentially and universally definable, and the theory of $G$ is model complete in $\mathfrak{L}_{oag}$ together with the predicates $j1_{i}+H_{i}$, where $i\in I$ and $j\in\mathbb{Z}$. \\
Also, we note that 
$$u\in j1_{i}+H_{i}\iff\exists r_{1},\ldots,r_{j}\in 1_{i}+H_{i}(u=r_{1}+\ldots+r_{j})$$ and $$u\in j1_{i}+H_{i}\iff\forall r_{1},\ldots,r_{j}\in 1_{i}+H_{i},\,\,r_{1}+\ldots+r_{j}-u\in H_{i}.$$  So, for each $j\in\mathbb{Z}$, the predicate $j1_{i}+H_{i}$ is existentially and universally definable using $1_{i}+H_{i}$ and $H_{i}$. Hence, for every $i\in I$, it suffices to add the predicates $j1_{i}+H_{i}$ for $j=0,1$ to $\mathfrak{L}_{oag}$. \end{proof}

In the case that the order of the group is discrete, we have the following result.
\begin{lem} \label{LemmaModelCompleteQuotientFiniteSpines} Let $\Gamma$ be a discrete oag with finite spines, $(H_{i})_{i\in I}$ an enumeration of its definable convex subgroups, and consider the language $\mathfrak{L}_{oag}^{*}=\{0,+,-,\leq,(j1_{i}+H_{i})_{i\in I,j=0,1}\}$ in which $Th(\Gamma)$ is model complete. If $\Delta$ is its minimal convex subgroup, then $Th(\Gamma/\Delta)$ is model complete in $\mathfrak{L}_{oag}^{*}$.\end{lem}
\begin{proof}
It suffices to note that
$(1_{i}+H_{i})^{\Gamma/\Delta}=(1_{i}+\Delta)+H_{i}^{\Gamma}/\Delta$ and $H_{i}^{\Gamma/\Delta}=H_{i}^{\Gamma}/\Delta$. Thus, $\Gamma/\Delta$ is an oag with finite spines and is model complete in $\mathfrak{L}_{oag}^{*}$ by Proposition \ref{PropMCFS}.
\end{proof}

To prove the main result in the case of finite spines (Theorem \ref{TheoremModelCompletenessFiniteRamification}), we need in particular the following general result on elementary embeddings of saturated oags with the smallest convex subgroup generated by the minimal positive element.
\begin{prop} \label{MCquotients}
Let $\Gamma,\Gamma^{'}$ $\aleph_{1}$-saturated discrete ordered abelian groups with minimal positive elements $1_{\Gamma}$ and $1_{\Gamma^{'}}$, respectively. Then $$\Gamma\preceq\Gamma^{'}\Longrightarrow\Gamma/\langle1_{\Gamma}\rangle_{conv}\preceq\Gamma^{'}/\langle1_{\Gamma^{'}}\rangle_{conv}$$
\end{prop}
\begin{proof}
Let $\Delta:=\langle1_{\Gamma}\rangle_{conv}$ and $\Delta^{'}:=\langle1_{\Gamma^{'}}\rangle_{conv}$.
Let $a_{1},\ldots,a_{n}\in\Gamma/\Delta$, then there are $b_{1},\ldots,b_{n}\in\Gamma$ such that $a_{i}=b_{i}+\Delta$ for $i=1,\ldots,n$. Set $A=\{a_{1},\dots,a_{n}\}$ and $B=\{b_{1}\,\dots,b_{n}\}$. Since $\Gamma\preceq\Gamma^{'}$ and they are $\aleph_{1}$-saturated, then $\Gamma\equiv_{B}\Gamma^{'}$, i.e. there exists a back-and-forth system of partial elementary maps $$I=\left\{f: C\longrightarrow D\,\vert\,\, C\subseteq_{fg}\Gamma, D\subseteq_{fg}\Gamma^{'}, B\subseteq C,f(b_{i})=b_{i}\right\}.$$ We want to show that $$\Gamma/\Delta\equiv_{A}\Gamma^{'}/\Delta^{'}.$$ 
Consider $$I^{'}=\left\{f_{\Delta}\,\vert\,f\in I\right\},$$ where $f_{\Delta}:(C+\Delta)/\Delta\longrightarrow (D+\Delta^{'})/\Delta^{'}$ is the map such that $(f_{\Delta})(c+\Delta)=f(c)+\Delta^{'}$.
\\
First of all, we show that this map is well defined. Notice that, since $\Gamma\preceq\Gamma^{'}$, then every partial elementary map $f\in I$ we have $f(1_{\Gamma})=1_{\Gamma^{'}}$. Thus, if $c+\Delta=d+\Delta\in\Gamma/\Delta$, then the following equivalences hold $$c-d\in\Delta\iff c-d=n1_{\Gamma}\iff f(c-d)=nf(1_{\Gamma})\iff f(c)-f(d)\in\Delta^{'},$$ which implies that $f_{\Delta}\in I^{'}$ is well defined.
Thus, $I^{'}$ is a back-and-forth system of elementary maps among finitely generated substructures of $\Gamma/\Delta$ and $\Gamma^{'}/\Delta^{'}$. Indeed, if $c+\Delta\in\Gamma/\Delta$ and  $c\notin (C+\Delta)/\Delta$, there is $$g^{'}\in I^{'}:(C+\Delta)/\Delta\cup\{c+\Delta\}\longrightarrow (D+\Delta^{'})/\Delta^{'}\cup\{g(c)+\Delta^{'}\},$$ where $g\in I$ is the extension of $f$ to $C\cup \{c\}$ and thus $g^{'}=g_{\Delta}$.
So $I^{'}$ is a back and forth system that fixes $a_{1},\ldots,a_{n}$,
and so $\Gamma/\Delta\equiv_{A}\Gamma^{'}/\Delta^{'}$.
Since $a_{1},\ldots,a_{n}$ have been chosen arbitrarily, we have $\Gamma/\Delta\preceq\Gamma^{'}/\Delta^{'}$.

 \end{proof}

Note that in the case of the latter proposition, the equivalence relation defined by the quotient is not definable without parameters. Indeed, otherwise, the thesis would have been immediate.

\section{Preliminaries on Henselian valued fields} \label{hvf}
We recall the notion of coarsening, a commonly used reduction method for valued fields.
We refer to \cite{EP05} and \cite{Cetraro2012_VDD_VF}.
Let $K$ be a valued field with value group $\Gamma_{K}$, valuation ring $O_{K}$, maximal ideal $M_{K}$, and residue field $k$.
In this section, we assume that the valued field $K$ has mixed characteristic $(0,p)$. We take the smallest convex subgroup $\Delta$ of $\Gamma_{K}$ containing $v(p)$ and consider the quotient $\Gamma_{K}/\Delta$. By the convexity of $\Delta$, this is also an ordered abelian group with the order induced by the order of $\Gamma_{K}$. The field $K$ carries a valuation which is the composition of $v$ with the canonical surjection $\pi:\Gamma_{K}\longrightarrow\Gamma_{K}/\Delta$. We denote this valuation by $\dot{v}:K\longrightarrow\Gamma_{K}/\Delta\cup\{\infty\}$, and we will refer to it as the \textit{coarse valuation} corresponding to $v$. The valuation ring associated to $\dot{v}$ is the set $$\{x\in K\,\vert\,\exists\delta\in\Delta\,v(x)\geq\delta\}$$ and its maximal ideal is $$\{x\in K\,\vert\,\forall\delta\in\Delta\,v(x)>\delta\},$$ which is contained in $\mathcal{M}_{K}$. We denote its residue field by $\mathring{K}$, and we call it the \textit{core field of} $K$ with respect to the valuation $v$, using the notation of Prestel and Roquette in \cite{PR84}. Note that $\mathring{K}$ has characteristic zero. We denote the valued field $(K,\dot{v},\mathring{K},\Gamma/\Delta)$ by $\dot{K}$. The core field has a valuation $\mathring{v}$ over the value group $\Delta$, defined by $\mathring{v}(x+ \mathcal{M}_{\dot{K}})=v(x)$. Its valuation ring is $O_{K}/\mathcal{M}_{\dot{K}}$ and its maximal ideal is $\mathcal{M}_{K}/\mathcal{M}_{\dot{K}}$.$\\$ We thus have the following picture:

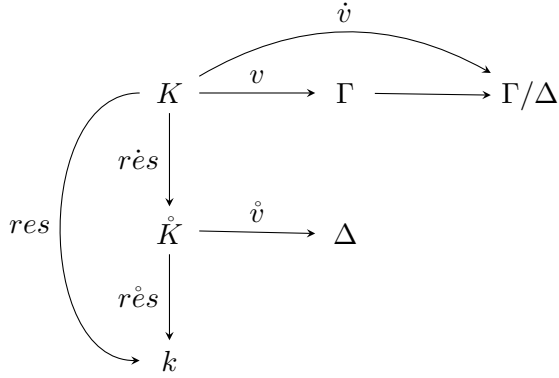
\vspace{2 mm}
\begin{tikzpicture}
  \matrix (m) [matrix of math nodes,row sep=3em,column sep=4em,minimum width=2em]
  {
     K & \Gamma & \Gamma/\Delta \\
     \mathring{K} & \Delta \\
     k \\};
  \path[-stealth]
    (m-1-1) edge node [left] {$\dot{res}$} (m-2-1)
            edge node [above] {$v$} (m-1-2)
	edge [bend left] node [above] {$\dot{v}$} (m-1-3) 
	edge [bend right=90] node [left] {$res$} (m-3-1)
    (m-2-1) edge node [left] {$\mathring{res}$} (m-3-1)
	edge node [above]{$\mathring{v}$} (m-2-2)
    (m-1-2) edge node [above] {} (m-1-3);
   
\end{tikzpicture}

\begin{defn} Let $(K,v,k,\Gamma)$ be a valued field of mixed characteristic $(0,p)$. We say that $K$ has \textit{finite ramification} $e$, if there is a natural number $e\geq1$ such that $e=\left\vert\left\{\gamma\in\Gamma: 0<\gamma\leq v(p)\right\}\right\vert$.  \end{defn}
\begin{lem}\cite[p. 27]{PR84} \label{ramificationcorefield} The ramification index of the core field $\mathring{K}$ with respect to $v_{0}$ and that of the field $K$ with respect to $v$ are the same. \end{lem}
We recall the next result from \cite{MD16}, 
in which the authors show that Julia Robinson's formula (\cite{RJ65}) $$\exists y(1+px^{2}=y^{2}),$$ which defines the valuation ring of the field of $p$-adic numbers $\mathbb{Q}_{p}$ ($p\neq 2$), can be used to define the valuation ring in a more general context.
\begin{prop} \label{valdefn} Let $K$ be an henselian valued field of mixed characteristic $(0,p)$ and ramification index $e$. Let $n>e$ be an integer coprime with $p$, then \begin{itemize} \item the valuation ring is existentially definable by the formula $\exists y(1+px^{n}=y^{n})$; \item the maximal ideal is existentially definable by the formula $\exists y(1+\frac{1}{p}x^{n}=y^{n})$, thus the valuation ring is universally definable. \end{itemize} \end{prop}
The following lemma from \cite{MD16} was proved for henselian mixed characteristic valued fields with the same finite ramification $e\geq1$ and valued in a $\mathbb{Z}$-group. Here we recall the proof in the more general case of value groups with a minimal positive element.

\begin{lem}\label{extvf} Let $(K_{1},O_{1})$ and $(K_{2},O_{2})$ be henselian valued fields of mixed characteristic $(0,p)$, such that $K_{1}\subseteq K_{2}$. If they have the same ramification index $0<e<\infty$, then $$O_{K_{2}}\cap K_{1}=O_{K_{1}}.$$ \end{lem}
\begin{proof} Note that since both valued fields have a finite ramification index, their value groups have a minimal element $1$. For $i=1,2$, the maximal ideal is $$\mathcal{M}_{K_{i}}=\{x\in K:x^{e}p^{-1}\in O_{K_{i}}\}.$$ Indeed, since $v(p)=e$, $$x\in\mathcal M_{K_{i}}\,\,\text{if and only if}\,\,v(x^{e}p^{-1})=ev(x)-e=e(v(x)-1)\geq0,$$ and so if and only if $x^{e}p^{-1}\in O_{K_{i}}$. From this observation and the fact that the valuation rings $O_{K_{1}}$ and $O_{K_{2}}$ are definable by the same existential formula (it depends only on $e$ and $p$), we deduce that the maximal ideals $\mathcal{M}_{K_{1}}$ and $\mathcal{M}_{K_{2}}$ are both defined by the existential formula $\exists y(1+p(x^{e}p^{-1})^{n}=y^{n})$. It follows that 
$O_{K_{1}}\subseteq O_{K_{2}}\cap K_{1}$. For the other inclusion suppose that there exists an element $\beta\in O_{K_{2}}\cap K_{1}$ and $\beta\notin O_{K_{1}}$. Then we have $\beta^{-1}\in O_{K_{1}}$, hence $\beta^{-1}\in O_{K_{2}}$, so $\beta$ is a unit in $O_{K_{2}}$. Since $\beta\notin O_{K_{1}}$, we have that $\beta^{-1}\in\mathcal{M}_{K_{1}}$, hence $\beta^{-1}\in\mathcal{M}_{K_{2}}$ and so we get a contradiction. 
\end{proof}

From the previous lemma, we get that $(K_{1},O_{1})\subseteq(K_{2},O_{2})$ is an extension of valued field. Hence, there is a natural inclusion of the residue field (resp. value group) of $K_{1}$ into the residue field (resp. value group) of $K_{2}$.

\subsection{AKE transfer principle}
We recall the well-known transfer principle due to Ax and Kochen \cite{AxK65} for equicharacteristic henselian valued fields.
\begin{thm}[Ax-Kochen, Ershov] \label{AKE} Let $K_{1},K_{2}$ be two equicharacteristic $0$ henselian valued fields in the language $\mathfrak{L}_{vf}$. Then \begin{itemize}
\item $K_{1}\equiv_{\mathfrak{L}_{vf}} K_{2}$ if and only if $\Gamma_{1}\equiv_{\mathfrak{L}_{oag}}\Gamma_{2}$ and $k_{1}\equiv_{\mathfrak{L}_{ring}}k_{2}$;
\item Suppose $K_{1}\subseteq_{\mathfrak{L}_{vf}}K_{2}$. Then $K_{1}\preceq K_{2}$ if and only if $\Gamma_{1}\preceq_{\mathfrak{L}_{oag}}\Gamma_{2}$ and $k_{1}\preceq_{\mathfrak{L}_{ring}}k_{2}$;\end{itemize}\end{thm}

In \cite[Appendix A]{R17} Rideau gives a useful analysis of relative quantifier elimination and recalls the notion of resplendent relative quantifier elimination, that is a relative elimination of quantifiers that is preserved when considering an enrichment of the language for some sorts. Applying this argument to valued fields, we recall the following result which states that the Ax-Kochen/Ershov principle works resplendently.
\begin{thm} \label{AKErespl} Let $K^{*}=(K,\Gamma_{K}^{*},k_{K}^{*})$ and $L^{*}=(L,\Gamma_{L}^{*},k_{L}^{*})$ be henselian valued fields of equicharacterisic $0$ in $\mathfrak{L}_{vf}^{*}=(\mathfrak{L}_{ring},\mathfrak{L}_{oag}^{*},\mathfrak{L}_{ring}^{*},v,res)$, where $\mathfrak{L}_{oag}^{*}$ is an expansion of $\mathfrak{L}_{oag}$ and $\mathfrak{L}_{ring}^{*}$ is an expansion of $\mathfrak{L}_{ring}$ Then \begin{itemize}
\item $K^{*}\equiv L^{*}$ if and only if $k_{K}^{*}\equiv k_{L}^{*}$ and $\Gamma_{K}^{*}\equiv\Gamma_{L}^{*}$.
\item If $K^{*}\subseteq L^{*}$ then $K^{*}\preceq L^{*}$ if and only if $k_{K}^{*}\preceq k_{L}^{*}$ and $\Gamma_{K}^{*}\preceq\Gamma_{L}^{*}$.\end{itemize} \end{thm} 

\subsection{Complete valued fields} \label{saturation}
The crucial step in the proof of our main result (Theorem \ref{TheoremModelCompletenessFiniteRamification}) is to show, under some conditions, an elementary embedding in the language of rings of two henselian valued fields $K_{1}$ and $K_{2}$. Thus, without loss of generality, we can assume $K_{1},\, K_{2}$ to be $\aleph_{1}$-saturated. Indeed, consider their ultrapowers $\prod_{D}K_{1}$ and $\prod_{D}K_{2}$ over a non-principal ultrafilter $D$. We have that if $\prod_{D}K_{1}\preceq\prod_{D}K_{2}$ then, since both $K_{1}$ and $K_{2}$ are elementary embedded into the respective ultrapowers, we have also $K_{1}\preceq K_{2}$.
In order to work in a saturated context, we recall the notion of pseudo completeness for a valued field, introduced by Ostrowski and Kaplansky in \cite{O17} and \cite{Kapl42}, respectively. This notion plays a key role in the study of maximal valued fields, i.e. valued fields with no immediate extensions.
\begin{defn}Let $(K,v)$ be a valued field. A sequence $(c_{\alpha})_{\alpha\in\lambda}$, indexed by some limit ordinal $\lambda$, is \textit{pseudo-Cauchy} if there is $\alpha_{0}<\lambda$ such that for all $\alpha_{0}<\alpha'<\alpha*<\alpha''<\lambda$, $$v(c_{\alpha''}-c_{\alpha*})>v(c_{\alpha*}-c_{\alpha'}).$$ \end{defn}
\begin{defn} Let $(K,v)$ be a valued field, $(c_{\alpha})_{\alpha\in\lambda}$ a pseudo-Cauchy sequence and $c\in K$. Then $(c_{\alpha})_{\alpha\in\lambda}$ \textit{pseudo converges} to $c$, if there is $0<\alpha_{0}<\lambda$ such that for every $\alpha_{0}<\alpha'<\alpha''<\lambda$, $$v(c-c_{\alpha''})>v(c-c_{\alpha'}).$$ In this case we say that $c$ is a \textit{pseudo-limit (PL)} of $(c_{\alpha})_{\alpha\in\lambda}$. \end{defn}
\begin{defn}
A valued field is $\lambda$\textit{-pseudo complete} if all pseudo-Cauchy sequences $(c_{\alpha})_{\alpha\in\lambda}$ have a pseudo-limit in $K$.
\end{defn}
The next is a well-known fact.
\begin{fact} \label{saturationpseudocomplete} An $\aleph_{1}$-saturated valued field is $\omega$-pseudo complete. \end{fact}
Let $K$ be a mixed characteristic valued field and consider the valued fields $\dot{K}$ and $\mathring{K}$ obtained by coarsening with their respective valuations (Section \ref{hvf}). We have the following
\begin{lem} \label{coarseningpreservespseudocompleteness} Let $\lambda$ be any ordinal. If $K$ is $\lambda$-pseudo complete then $\dot{K}$ and $\mathring{K}$ with the respective valuations are $\lambda$-pseudo complete.\end{lem}
\begin{proof} The valuation over $\dot{K}$ is given by the composition of $v$ and the projection map $\pi:\Gamma\longrightarrow\Gamma/\Delta$. Since $\pi$ preserves the order and the valued fields have the same domain $K$, then pseudo-Cauchy sequences and pseudo-limits are preserved. Thus $\dot{K}$ is $\lambda$-pseudo complete. 
Now, let $(x_{\alpha}+M_{K})_{\alpha<\lambda}$ be a pseudo-Cauchy sequence of elements in $\mathring{K}$, where $M_{K}$ is the maximal ideal of $K$ with respect to the valuation $v$. Then for $\alpha_{0}<\nu<\alpha<\mu<\lambda$, \begin{align*}&\mathring{v}(x_{\nu}+M_{K}-(x_{\alpha}+M_{K}))<\mathring{v}(x_{\alpha}+M_{K}-(x_{\mu}+M_{K}))\Longleftrightarrow \\ &\mathring{v}(x_{\nu}-x_{\alpha}+M_{K})<\mathring{v}(x_{\alpha}-x_{\mu}+M_{K})\Longleftrightarrow \\ & v(x_{\nu}-x_{\alpha})<v(x_{\alpha}-x_{\mu}),\end{align*} thus $(x_{\alpha}+M_{K})_{\alpha<\lambda}$ is pseudo-Cauchy in $(\mathring{K},\mathring{v})$ if and only if $(x_{\alpha})_{\alpha<\lambda}$ is pseudo-Cauchy in $(K,v)$. In the same way, it is easy to prove that $c$ is a pseudo-limit for $(x_{\alpha})_{\alpha<\lambda}$ in $(K,v)$ if and only if $c+M_{K}$ is a pseudo-limit for the respective sequence in $(\mathring{K},\mathring{v})$. This proves that if $(K,v)$ is $\lambda$-pseudo complete then the core field $(\mathring{K},\mathring{v})$ is $\lambda$-pseudo complete. \end{proof}
\begin{rem}\label{coreCauchycomplete}Note that, since $\mathring{K}$ has value group isomorphic to $\mathbb{Z}$, it is actually Cauchy complete.\end{rem}
For complete valued fields the following characterization holds (\cite[Sec 5, Theorem 4]{S79}).
\begin{thm} \label{wittringchar} Let $K$ be a complete valued field of mixed characteristic $(0,p)$ with value group isomorphic to $\mathbb{Z}$, and let $e$ be its ramification index. Then $K$ is a finite extension of $W(k)$, that is the fraction field of the Witt ring over $k$. In particular, $K=W(k)(\pi)$, where $\pi$ is a root of an Eisenstein polynomial over $W(k)$, i.e. a polynomial $$x^{e}+a_{e-1}x^{n-1}+\dots+a_{1}x+a_{0}\in W(k)[x]$$ such that for all $i\in\{0,\dots,e-1\}$, $p|a_{i}$ and $p^{2}\nmid a_{0}$, i.e. $a_{0}$ has minimal positive valuation. $\\$ Conversely, a root of such a polynomial, defines a totally ramified extension of $W(k)$ of degree $e$, and that root is an element with minimal positive valuation in $K$.  \end{thm}

\section{A one-sorted language for model completeness in the case of value group with finite spines}
In the proof of Theorem \ref{TheoremModelCompletenessFiniteRamification} we will use the Ax-Kochen/Ershov principle (Theorem \ref{AKE}) and in particular its property of resplendence (Theorem \ref{AKErespl}). We shall work with two languages $\mathfrak{L}^{*}_{vf}$ and $\mathfrak{L}^{**}_{vf}$ that expand the usual language of valued fields $\mathfrak{L}_{vf}$ respectively on the group sort and on the field sort. To understand the relation between these two languages, we give the following

\begin{lem}\label{bi-interpretablelanguages}
Consider $\mathfrak{L}_{oag}^{*}=\mathfrak{L}_{oag}\cup\{j1_{i}+H_{i}\}_{i\in I,j=0,1}$ and the expansion of the language of valued fields $$\mathfrak{L}_{vf}^{*}=(\mathfrak{L}_{ring},\mathfrak{L}_{oag}^{*},\mathfrak{L}_{ring},v,res).$$ Let $(K,\Gamma^{*},k,v,res)$ be an $\mathfrak{L}_{vf}^{*}$-structure. Define predicates $A_{i,j}$, for $i\in I$ and $j=0,1$, in the sort $K$, such that $$A_{i,j}^{K}=\left\{a\in K\,\vert\,v(a)\in j1_{i}+{H_{i}}\right\},$$ and consider the language $$\mathfrak{L}_{vf}^{**}=(\mathfrak{L}_{ring}^{*},\mathfrak{L}_{oag},\mathfrak{L}_{ring},v,res),$$ where $\mathfrak{L}^{*}_{ring}=\mathfrak{L}_{ring}\cup\{A_{i,j}\}_{i\in I,j=0,1}$. Then $\mathfrak{L}^{*}_{vf}$ and $\mathfrak{L}^{**}_{vf}$ are bi-interpretable. \end{lem}
\begin{proof}It follows from the definition of the languages. \end{proof}

In particular, we will use that the languages $\mathfrak{L}_{vf}^{*}$ and $\mathfrak{L}_{vf}^{**}$ preserve substructures and elementary extensions. We note that the bi-interpretability of the languages has the following consequences.
\begin{rem} \label{embedd} Let $$\mathcal{K}_{1}=(K_{1},\Gamma_{1}^{*},k_{1},v_{1},res_{1}),$$ $$\mathcal{K}_{2}=(K_{2},\Gamma_{2}^{*},k_{2},v_{2},res_{2})$$  be two $\mathfrak{L}_{vf}^{*}$-structures and $$\widetilde{\mathcal{K}_{1}}=(K_{1}^{*},\Gamma_{1},k_{1},v_{1},res_{1}),$$ $$\widetilde{\mathcal{K}_{2}}=(K_{2}^{*},\Gamma_{2},k_{2},v_{2},res_{2})$$ the corresponding $\mathfrak{L}_{vf}^{**}$-structures. Then
$$\mathcal{K}_{1}\subseteq_{\mathfrak{L}_{vf}^{*}}\mathcal{K}_{2}\iff\widetilde{\mathcal{K}_{1}}\subseteq_{\mathfrak{L}_{vf}^{**}}\widetilde{\mathcal{K}_{2}}.$$ \end{rem}
\begin{proof}
It suffices to show that the thesis holds considering the restrictions of the languages to the first two sorts and the function symbol defined between them, that is
$$(K_{1},\Gamma_{1}^{*},v_{1})\subseteq_{\mathfrak{L}_{{vf}^{*}_{\restriction_{ K,\Gamma^{*}}}}}(K_{2},\Gamma_{2}^{*},v_{2})\iff(K_{1}^{*},\Gamma_{1},v_{1})\subseteq_{\mathfrak{L}_{{vf}^{**}_{\restriction_{K^{*},\Gamma}}}}(K_{2}^{*},\Gamma_{2},v_{2}).$$ 
First note that $(K_{1},\Gamma_{1}^{*},v_{1})\subseteq_{\mathfrak{L}_{{vf}^{*}_{\restriction_{ K,\Gamma^{*}}}}}(K_{2},\Gamma_{2}^{*},v_{2})$ if and only if \newline $K_{1}\subseteq_{\mathfrak{L}_{ring}}K_{2}$, $\Gamma_{1}^{*}\subseteq_{\mathfrak{L}_{oag}^{*}}\Gamma_{2}^{*}$ and $v_{2\restriction_{K_{1}}}=v_{1}$. Analogously for $(K_{1}^{*},\Gamma_{1},v_{1})$ and $(K_{2}^{*},\Gamma_{2},v_{2})$.\\
Thus, suppose that $$(K_{1}^{*},\Gamma_{1})\nsubseteq_{\mathfrak{L}_{{vf}^{**}_{\restriction_{K^{*},\Gamma}}}}(K_{2}^{*},\Gamma_{2}).$$ \\ Clearly, if $\Gamma_{1}\nsubseteq_{\mathfrak{L}_{oag}}\Gamma_{2}$ then $\Gamma_{1}\nsubseteq_{\mathfrak{L}_{oag}^{*}}\Gamma_{2}$, and so we have $$(K_{1},\Gamma_{1}^{*},v_{1})\nsubseteq_{\mathfrak{L}_{{vf}^{*}_{\restriction_{K,\Gamma^{*}}}}}(K_{2},\Gamma_{2}^{*},v_{2}).$$ \\ Thus, suppose $K_{1}^{*}\nsubseteq_{\mathfrak{L}_{ring}^{*}}K_{2}^{*}$ with $A_{i,j}^{K_{1}^{*}}\neq A_{i,j}^{K_{2}^{*}}\cap K_{1}^{*}$ for some $i\in I$ and $j\in\{0,1\}$. This means that there is $x\in K_{1}^{*}$ such that $v_{1}(x)\in\,j1_{i}+H_{i}$ and $v_{2}(x)\notin j1_{i}+H_{i}$.
This happens if and only if $$(j1_{i}+H_{i})^{\Gamma_{1}^{*}}\neq(j1_{i}+H_{i})^{\Gamma_{2}^{*}}\cap\Gamma_{1}^{*}/H_{i},$$ since $v_{2}$ is an extension of $v_{1}$ to $K_{2}^{*}$. Thus $\Gamma_{1}^{*}\nsubseteq_{\mathfrak{L}_{oag}^{*}}\Gamma_{2}^{*}$ and so $\widetilde{\mathcal{K}_{1}}\nsubseteq_{\mathfrak{L}_{vf}^{*}}\widetilde{\mathcal{K}_{2}}$. \\
\end{proof}

\begin{cor} \label{elemb} Let $\mathcal{K}_{1},\mathcal{K}_{2}$ be $\mathfrak{L}^{*}_{vf}$-structures such that $\mathcal{K}_{1}\subseteq\mathcal{K}_{2}$. Then $$\mathcal{K}_{1}\preceq_{\mathfrak{L}^{*}_{vf}}\mathcal{K}_{2}\Longleftrightarrow\widetilde{\mathcal{K}_{1}}\preceq_{\mathfrak{L}^{**}_{vf}}\widetilde{\mathcal{K}_{2}}$$\end{cor}
\begin{proof} Suppose $\mathcal{K}_{1}\preceq_{\mathfrak{L}^{*}_{vf}}\mathcal{K}_{2}$. Then by Remark \ref{embedd}, $\widetilde{\mathcal{K}_{1}}\subseteq\widetilde{\mathcal{K}_{2}}$ in $\mathfrak{L}_{vf}^{**}$ and suppose for a contraddiction that $\widetilde{\mathcal{K}_{1}}\not\preceq\widetilde{\mathcal{K}_{2}}$. Then there exists an $\mathfrak{L}_{vf}^{**}$-formula $\widetilde{\phi}(x)=\exists y\psi(x,y)$ with $\psi$ quantifier free and $a\in K_{1}^{*}\cup\Gamma_{1}\cup k_{1}$ such that $$\widetilde{\mathcal{K}_{1}}\models\widetilde{\phi}(a)\,\,\text{and}\,\,\widetilde{\mathcal{K}_{2}}\not\models\widetilde{\phi}(a).$$ By the bi-interpretability of the languages (Lemma \ref{bi-interpretablelanguages}), there exists an $\mathfrak{L}_{vf}^{*}$-formula $\phi$ and $b\in K_{2}\cup\Gamma_{2}^{*}\cup k_{2}$ such that $$\mathcal{K}_{1}\models\phi(a)\,\,\text{and}\,\,\mathcal{K}_{2}\not\models\phi(a).$$\end{proof}
 
We can now give the proof of the main theorem of this section.
\begin{thm}\label{TheoremModelCompletenessFiniteRamification} Let $K$ be an Henselian valued field of mixed characteristic $(0,p)$, finite ramification $e\geq1$, residue field $k$ and value group $\Gamma$ with finite spines. Let $(H_{i})_{i\in I}$ be an enumeration of the definable convex subgroups of $\Gamma$. If the theory of the residue field is model complete in the language of rings, then the theory of $K$ is model complete in the language $\mathfrak{L}_{Ring}^{*}=\{0,+,1,\cdot, A_{i,j}\}$ where $(A_{i,j})_{i\in I,j=0,1}$ are predicates such that $$A_{i,j}^{K}=\{a\in K\,\vert\, v(a)\in\,j1_{i}+H_{i}\}.$$\end{thm}
\begin{proof} We follow the structure of the proof of Theorem 1 in \cite{MD16}. We start working with the three sorted language $\mathfrak{L}_{vf}^{**}=(\mathfrak{L}^{*}_{Ring},\mathfrak{L}_{oag},\mathfrak{L}_{ring},v,res)$ that expands the language $\mathfrak{L}_{vf}$ on the field sort, with the language $\mathfrak{L}^{*}_{Ring}$ defined in the statement. Let $$(K_{1},\Gamma_{1},k_{1},v_{1},res_{1}),$$ $$(K_{2},\Gamma_{2},k_{2},v_{2},res_{2})$$ be two models of $Th(K)$ in $\mathfrak{L}^{**}_{vf}$ such that $K_{1}\subseteq K_{2}$ as expanded fields in $\mathfrak{L}^{*}_{Ring}$, and with $\Gamma_{1},\,\Gamma_{2}$ and $k_{1},\,k_{2}$ their value groups and residue fields, respectively. We want to show that $K_{1}\preceq K_{2}$ in $\mathfrak{L}^{*}_{Ring}$. By Section \ref{saturation}, we may assume $K_{1},\,K_{2}$ to be $\aleph_{1}$-saturated. Lemma \ref{extvf} ensures that that $(K_{1},v_{1})\subseteq(K_{2},v_{2})$ is an embedding of valued fields, thus $$(K_{1},\Gamma_{1},k_{1},v_{1},res_{1})\subseteq_{\mathfrak{L}^{**}_{vf}}(K_{2},\Gamma_{2},k_{2},v_{2},res_{2}).$$ By Remark \ref{embedd}, this is an embedding also in $\mathfrak{L}^{*}_{vf}$, thus $\Gamma_{1}$ and $\Gamma_{2}$ are models of $Th(\Gamma)$ in $\mathfrak{L}^{*}_{oag}$ such that $\Gamma_{1}\subseteq\Gamma_{2}$. Since $Th(\Gamma)$ is model complete in $\mathfrak{L}^{*}_{oag}$ by Proposition \ref{PropMCFS}, the embedding is elementary.
For both valued fields, we consider the coarse valuation and the resulting decomposition (see Section \ref{hvf}). Thus we have the two equicharacteristic zero valued fields $\dot{K}_{i}=(K_{i},\dot{v}_{i})$, $i=1,2$, valued on $\Gamma_{i}/\Delta$ and with residue fields the core fields $\mathring{K}_{i}$, $i=1,2$, on which a mixed characteristic valuation $\mathring{v}_{i}$ is defined with values on $\Delta$ and residue fields $k_{i}$, $i=1,2$ respectively. Note that, by Lemma \ref{MCquotients}, $\Gamma_{1}\preceq\Gamma_{2}$ implies $\Gamma_{1}/\Delta\preceq\Gamma_{2}/\Delta$. So, if also $\mathring{K}_{1}\preceq \mathring{K}_{2}$, then by the resplendent AKE principle for equicharacteristic zero henselian valued fields (Theorem \ref{AKErespl}), we have that $\dot{K}_{1}\preceq\dot{K}_{2}$ in $\mathfrak{L}^{*}_{vf}$. Remark \ref{elemb} implies that if $\dot{K}_{1}\preceq\dot{K}_{2}$ in $\mathfrak{L}^{*}_{vf}$, then $\dot{K}_{1}\preceq\dot{K}_{2}$ in $\mathfrak{L}^{**}_{vf}$. Thus $K_{1}\preceq K_{2}$ in $\mathfrak{L}_{Ring}^{*}$ as expanded rings and since the valuation is existentially and universally definable by Proposition \ref{valdefn}, 
their embedding is elementary in $\mathfrak{L}_{Ring}^{*}$ as valued fields. This will conclude the proof. 
Thus, it remains to show that $\mathring{K}_{1}\preceq \mathring{K}_{2}$, assuming $k_{1}\preceq k_{2}$ in $\mathfrak{L}_{ring}$.

Since $K_{1}$ and $K_{2}$ are $\aleph_{1}$-saturated, by Lemma \ref{saturationpseudocomplete} they are $\omega$-pseudo complete and by Lemma \ref{coarseningpreservespseudocompleteness} also the coarse fields $\dot{K}_{1}$ and $\dot{K}_{2}$ and the core fields $\mathring{K}_{1},\mathring{K}_{2}$ are $\omega$-pseudo complete. In particular, the core fields are actually Cauchy complete, since their value groups are canonically isomorphic to $\mathbb{Z}$. Also, by Remark \ref{ramificationcorefield} we know that the core field has ramification index $e$. Theorem \ref{wittringchar} gives a characterization of complete fields with fixed ramification in terms of finite extensions of the field of fractions of a Witt ring. Indeed, we have $\mathring{K}_{1}=W(k_{1})(\pi)$ for some $\pi\in\mathring{K}_{1}$ with minimal positive valuation and which is a root of a polynomial $$E(x)=x^{e}+c_{e-1}x^{e-1}+\dots+c_{1}x+c_{0}$$ that is Eisenstein over $W(k_{1})$. Thus, $$c_{i}\in M_{W(k_{1})}$$ and $$c_{0}\in M_{W(k_{1})}\setminus\,M^{2}_{W(k_{1})}.$$ $\\$
\textbf{First claim.} The polynomial $E(x)$ is Eisenstein also over $W(k_{2})$ and $\mathring{K}_{2}=W(k_{2})(\pi)$.$\\$
\textit{Proof of the first claim.} First note that, if $val$ is the valuation over $W(k_{1})$, then $$c_{i}\in M_{W(k_{1})}\,\text{if and only if}\,\,c_{i}^{e}p^{-1}\in O_{W(k_{1})}.$$ Indeed, $$val(c_{i}^{e}p^{-1})=e\,val(c_{i})-val(p)=e\,val(c_{i})-e=e(val(c_{i})-1);$$ and that having $c_{0}\in M_{W(k_{1})}\setminus\, M^{2}_{W(k_{1})}$ is equivalent to say that $$c_{0}^{e}p^{-1}\in O_{W(k_{1})}$$ and $$c_{0}^{-e}p\in O_{W(k_{1})}.$$ Indeed, the latter means that $val(c_{0})\leq1$, so from both one has $val(c_{0})=1$.
Now, note that $W(k_{1})$ and $W(k_{2})$ as valued fields are unramified and $p$ is an element with minimal positive valuation in both $W(k_{1})$ and $W(k_{2})$. Then by Proposition \ref{valdefn}, their valuation rings are defined by the same existential formula, so we have that $$c_{i}\in M_{W(k_{2})}$$ and $$c_{0}\in M_{W(k_{2})}\setminus\, M^{2}_{W(k_{2})}.$$ Thus, $E(x)$ is an Eisenstein polynomial also on $W(k_{2})$ and $\pi$ is a root, so it is an element of $\mathring{K}_{2}$ with the minimal positive valuation. Therefore, the extension $W(k_{2})(\pi)$ has dimension $e$ over $W(k_{2})$, which is a subfield of $\mathring{K}_{2}$ of the same dimension and $\pi\in\mathring{K}_{2}$. Hence, $\mathring{K}_{2}=W(k_{2})(\pi)$. $\\$ $\\$
\textbf{Second claim.} It remains to show that the embedding of $W(k_{1})$ into $W(k_{2})$ is elementary, from which we deduce that also $$\mathring{K}_{1}=W(k_{1})(\pi)\longrightarrow W(k_{2})(\pi)=\mathring{K}_{2}$$ is elementary.
$\\$ \textit{Proof of the second claim.} Assume $k_{1}\preceq k_{2}$. Then $W(k_{1})$ and $W(k_{2})$ are unramified henselian valued fields with value group $\mathbb{Z}$ and residue fields $k_{1}$ and $k_{2}$, respectively. Thus, by Theorem \ref{AKE}, we have $W(k_{1})\preceq W(k_{2})$.
It remains to show that $\mathring{K}_{1}\preceq\mathring{K}_{2}$.
We can interpret $W(k_{i})(\pi)$ inside $W(k_{i})$ (for $i=1,2$) as follows. We identify $W(k_{i})(\pi)$ with $W(k_{i})^{e}$. On the $e$-tuples we define addition as the usual addition on vector spaces. Knowing that $E(x)$ is the minimal polynomial of $\pi$ over $W(K_{i})$ we can compute the multiplication by $\pi$ into $W(K_{i})(\pi)$. Thus, an $e\times e$-matrix $M_{\pi}$ is determined and it depends on the coefficients $c_{0},\dots,c_{e-1}$ of $E(x)$. We can define a multiplication over $W(k_{i})^{e}$ as follows
$$(x_{1},\dots,x_{e})\times(y_{1},\dots,y_{e})=(x_{1}I_{e}+x_{2}M_{\pi}+\dots+x_{e}M_{\pi}^{e-1})\begin{pmatrix} y_{1}\\ \vdots \\ y_{e} \end{pmatrix}$$
where $I_{e}$ is the identity $e\times e$-matrix.
Thus, we obtain that $\mathring{K}_{1}\preceq\mathring{K}_{2}$, and the theorem is proved. \end{proof}

Notice that in the proof we use the Witt rings over the residue fields without the hypothesis of perfectness. Indeed, we have the following
\begin{rem} \label{perfect} Let $k$ be a field. If $Th(k)$ is model complete in $\mathfrak{L}_{ring}$, then $k$ is perfect.\end{rem}
\begin{proof} If $char(k)=0$ the assert is obvious. Let $char(k)=p$ where $p>1$ is a prime and consider the field $k^{1/p}=k(a^{1/p}:\,a\in k)$. Then \begin{align*}f:&\,\,k^{1/p}\longrightarrow \,\,k \\ &\,\,\,\,\,x\,\,\,\,\,\,\mapsto \,\,\,x^{p}\end{align*} is an isomorphism of fields. Thus, $k^{1/p}\models Th(k)$, and in particular, $k\subseteq k^{1/p}$ as $\mathfrak{L}_{ring}$-structures. Thus, for $a\in k$, we have that $a$ satisfies the formula $\phi(y)=\exists x(x^{p}=y)$ in $k^{1/p}$, and since $k\preceq k^{1/p}$ by hypothesis, there exists a $p$-th root of $a$ in $k$. Thus, $k$ is perfect. \end{proof}

\section{A one-sorted language for model completeness in the case of value group elementarily equivalent to $\bigoplus_{i<\omega^{*}}{\mathbb{Z}}$}
\subsection{Model theory of $\bigoplus_{i<\omega^{*}}{\mathbb{Z}}$}\label{mtlpz}
Consider $\bigoplus_{i<\omega^{*}}{\mathbb{Z}}$ the lexicographic sum of $\mathbb{Z}$, such that $(\ldots,0,0,1)$ is the minimal positive element. The examples in Section \ref{examples} motivate the choice of this group. We want to extend the results of the previous section to the case of finitely ramified henselian valued fields valued in a group that is elementarily equivalent to $\bigoplus_{i<\omega^{*}}{\mathbb{Z}}$. For this purpose, we give a brief analysis of the model theory of $\bigoplus_{i<\omega^{*}}{\mathbb{Z}}$ as we have done for ordered abelian groups with finite spines. Schmitt's definition of the $n$-spine of an ordered abelian group (see Section \ref{OAGfs}) works for an arbitrary ordered abelian group in the most general case. In \cite{hlt23}, the authors identify some assumptions on an ordered abelian group under which the induced structure on the spine is that of a pure colored total order $(\Gamma,<, C_{\phi})_{\phi\in\mathfrak{L}_{oag}}$, where $C_{\phi}$ are unary predicates (the colours) for each formula $\phi\in\mathfrak{L}_{oag}$. For instance, assume that $(G_{\gamma})_{\gamma\in\Gamma}$ is a family of non-trivial archimedean ordered abelian groups such that for every $\gamma\in\Gamma$, the group $G_{\gamma}$ is either a discretely ordered group or a dense one, and where $(\Gamma,<)$ is an ordered set. Consider the Hahn product $$H=\{f\in\prod_{\gamma\in\Gamma}G_{\gamma}\,:\,supp(f)\,\text{is a well-ordered subset of}\,(\Gamma,<)\},$$ equipped with the lexicographic order. Then $H$ is an ordered abelian group with spine $(\Gamma,<, C_{\phi_{1}}, C_{\phi_{2}})_{\phi_{1},\phi_{2}\in\mathfrak{L}_{oags}}$, where $\phi_{1}$ and $\phi_{2}$ describe, respectively, discretely ordered groups and densely ordered groups and, for $i=1,2$, $$C_{\phi_{i}}(\gamma)\,\Longleftrightarrow\,G_{\gamma}\models\phi_{i}.$$ In the case of the lexicographic sum of $\mathbb{Z}$, one can easily note that the induced structure on the spine is $(\omega,<)$, and there are no colours. Thus, we start giving the following fact deduced from (\cite[Theorem 2.18]{hlt23}), which is obtained by a reduction of Cluckers and Halupczok's language for the relative quantifier elimination in ordered abelian groups (\cite{CH11}).
\begin{fact} \label{factlpz} Let $\mathfrak{L}$ be a language consisting of \begin{itemize}
\item the main sort $G$ with $+,-,0,<,\equiv_{m}\,(m\in\mathbb{N})$;
\item an auxiliary sort $\Gamma$ with $<,0,\infty,s:\Gamma\longrightarrow\Gamma$;
\item $val^{n}:G\longrightarrow\Gamma\,(n\in\mathbb{N},n\neq1)$,
\item an unary predicate $=^{\bullet}k_{\bullet}$ on $G$ for each $k\in\mathbb{Z}\setminus\{0\}$,
\item an unary predicate $\equiv^{\bullet}_{m}k_{\bullet}$ on $G$ for each $m\geq2\,and\,k\in\{1,\ldots,m-1\}$.
\end{itemize}
Then the theory of $G=\bigoplus_{i<\omega^{*}}{\mathbb{Z}}$ has quantifier elimination in $\mathfrak{L}$, where
\begin{itemize}
\item $\Gamma=\omega\cup\{\infty\}$,
\item $s(n)=n+1$,
\item for every $a\in G$, $val^{n}(a):=minsupp(a\,\mod{nG})$ if $a\notin nG$, \\ $val^{n}(a):=\infty$ otherwise (or equivalently $val^{n}(a)$ is the index $i$ of the largest convex subgroup $H_{i}$ such that $a\notin H_{i}+nG$),
\item for every $a\in G$, $a=^{\bullet}k_{\bullet}$ if $a+H_{i}$ is $k$ times the minimal element of $G/H_{i}$ for some $i\in\Gamma$,
\item for every $a\in G$, $a\equiv^{\bullet}_{m}k_{\bullet}$ if $a+H_{i}$ is congruent modulo $m$ to $k$ times the minimal element of $G/H_{i}$ for some convex subgroup $H_{i}$.
\end{itemize}
\end{fact}

We determine a language for the theory of $G$ to be model complete.
\begin{prop} \label{proplpz}The theory of $G=\bigoplus_{i<\omega^{*}}{\mathbb{Z}}$ is model complete in the one-sorted language $\mathfrak{L}^{\star}_{oag}$ consisting of
\begin{itemize}
    \item $+,-,0,<$ for the oag $G$,
    \item for every $n,m\in\mathbb{N}$ a relation symbol $|^{n,m}$ on $G$,
    \item for every $n,m\in\mathbb{N}$ a binary predicate $\overline{s}^{n,m}$ on $G$
    \item for every $a\in G$, $a=^{\bullet}1_{\bullet}$ if $a+H_{i}$ is the minimal element of $G/H_{i}$ for some convex subgroup $H_{i}$.
\end{itemize}\end{prop}
\begin{proof}
We define, for every $n,m\in\mathbb{N}$ the relation $|^{n,m}$ on $G$ as follows $$x|^{n,m}y\Longleftrightarrow val^{n}(x)<val^{m}(y)$$ to interpret the order of $\Gamma$ in $G$.
As well, define the binary relation $\overline{s}^{n,m}$ on $G$ as $$\overline{s}^{n,m}(x,y)\Longleftrightarrow s(val^{n}(x))=val^{m}(y)$$ in order to interpret the successor function of $\Gamma$ in $G$. Thus $$\overline{s}^{n,m}(x,y)=\left\{(x,y)\in G^{2}\vert\,\,\, x|^{n,m}y\,\wedge\,\forall z\in G\,\neg(x|^{n,l}z\wedge z|^{l,m}y),\text{for any}\,l\in\mathbb{N}\right\}.$$
Finally note that, for each $m\geq2$ and $k\in\{1,\ldots,m-1\},$ $$x\equiv_{m}^{\bullet}k_{\bullet}\Longleftrightarrow\exists b\in G\,(x-mb=^{\bullet}k_{\bullet})$$ and \begin{align*}x\not\equiv_{m}^{\bullet}k_{\bullet}\Longleftrightarrow &\exists b\in G\,\,x-mb=^{\bullet}(1+k)_{\bullet}\,\vee\,x-mb=^{\bullet}(2+k)_{\bullet}\,\vee\ldots \\ &\ldots\vee\,x-mb=^{\bullet}(m-1+k)_{\bullet}.\end{align*}
Also, for every $k\geq1$, \begin{equation}\begin{split}x=^{\bullet}k_{\bullet}\iff & \exists z,y\,(x=\underbrace{z+\ldots+z}_{k\;times}+y\,\wedge\,x|^{0,0}z\,\wedge z|^{0,0}x\,\wedge\,z=^{\bullet}1_{\bullet}\,\wedge \\ & \wedge\,y|^{0,0}x\,\lor\, y=0)\end{split}\end{equation} and \begin{equation}\begin{split}x=^{\bullet}k_{\bullet}\iff & \forall z\in G\,(x|^{0,0}z\,\wedge z|^{0,0}x\,\wedge\,z=^{\bullet}1_{\bullet})\,\rightarrow (x-(\underbrace{z+\ldots+z}_{k\;times}))\,\vert^{0,0}\,x.\end{split}\end{equation} Thus, for every $k$, the predicates $=^{\bullet}k_{\bullet}$ are existentially and universally definable using $=^{\bullet}1_{\bullet}$, so we can substitute them in the language in order to have the model completeness of the theory of $G$.
\end{proof}

\subsection{Value group elementarily equivalent to $\bigoplus_{i<\omega^{*}}{\mathbb{Z}}$}
In this section we give an analogue of Theorem \ref{TheoremModelCompletenessFiniteRamification} for henselian valued field of mixed charactestistic $(0,p)$, fixed ramification $e\geq1$ and value group elementarily equivalent to $\bigoplus_{i<\omega^{*}}{\mathbb{Z}}$.

\begin{thm} \label{TheoremModelCompletenessProductZ}
Let $K$ be an Henselian valued field of mixed characteristic $(0,p)$, finite ramification $e\geq1$, residue field and value group a model $G$ of $Th(\bigoplus_{i<\omega^{*}}{\mathbb{Z}})$. If the theory of the residue field is model complete in the language of rings, then the theory of $K$ is model complete in the language $\mathfrak{L}_{ring}^{\star}$ consisting of 
\begin{itemize}
   \item $+,0,\cdot,1$
   \item for every $n,m\in\mathbb{N}$ a relation symbol $||^{n,m}$,
   \item for every $n,m\in\mathbb{N}$ a binary predicate $\$^{n,m}$,
   \item an unary predicate $A$,
\end{itemize}
where \begin{itemize}
    \item for every $x,y\in K$, $x||^{n,m}y\Longleftrightarrow val^{n}(v(x))\leq val^{m}(v(y))$,
    \item for every $x,y\in K$, $\$^{n,m}(x,y)\Longleftrightarrow val^{m}(v(y))=s(val^{n}(v(x)))$,
    \item $A^{K}=\left\{x\in K\,\vert\,v(x)=^{\bullet}1_{\bullet}\right\}$.
\end{itemize} 
\end{thm}

The proof is similar to the proof of Theorem \ref{TheoremModelCompletenessFiniteRamification}. We only need to note, as in the previous section, the following lemma on bi-interpretability of the languages, in order to deal with different languages when applying the resplendent version of AKE Theorem (Theorem \ref{AKErespl}).
\begin{lem} \label{bi-interpretability2} Consider $\mathfrak{L}_{oag}^{\star}$ and the expansion of the language of valued fields $$\mathfrak{L}_{vf}^{\star}=(\mathfrak{L}_{ring},\mathfrak{L}_{oag}^{\star},\mathfrak{L}_{ring},v,res).$$ Let $(K,\Gamma^{\star},k,v,res)$ be an $\mathfrak{L}_{vf}^{\star}$-structure. Consider the language $$\mathfrak{L}_{vf}^{\star\star}=(\mathfrak{L}_{ring}^{\star},\mathfrak{L}_{oag},\mathfrak{L}_{ring},v,res),$$ where $\mathfrak{L}^{\star}_{ring}$ is the language defined in Theorem \ref{TheoremModelCompletenessProductZ}. Then, the languages $\mathfrak{L}_{vf}^{\star}$ and $\mathfrak{L}_{vf}^{\star\star}$ are bi-interpretable.\end{lem}
Also in this case, we use the following remark.
\begin{rem}\label{embedd2} Let $$\mathcal{K}_{1}=(K_{1},\Gamma_{1}^{\star},k_{1},v_{1},res_{1}),$$ $$\mathcal{K}_{2}=(K_{2},\Gamma_{2}^{\star},k_{2},v_{2},res_{2})$$  be two $\mathfrak{L}_{vf}^{\star}$-structures and $$\widetilde{\mathcal{K}_{1}}=(K_{1}^{\star},\Gamma_{1},k_{1},v_{1},res_{1}),$$ $$\widetilde{\mathcal{K}_{2}}=(K_{2}^{\star},\Gamma_{2},k_{2},v_{2},res_{2})$$ the corresponding $\mathfrak{L}_{vf}^{\star\star}$-structures. Then $$\mathcal{K}_{1}\subseteq_{\mathfrak{L}_{vf}^{\star}}\mathcal{K}_{2}\iff\widetilde{\mathcal{K}_{1}}\subseteq_{\mathfrak{L}_{vf}^{\star\star}}\widetilde{\mathcal{K}_{2}}.$$\end{rem}
\begin{proof} 
As in Remark \ref{elemb}, it suffices to prove that the thesis holds for the restrictions of the languages to the first two sorts and the function symbol $v$ defined between them.  $\\$
First note that $(K_{1},\Gamma_{1}^{\star},v_{1})\subseteq_{\mathfrak{L}_{{vf}^{\star}_{\restriction_{ K,\Gamma^{\star}}}}}(K_{2},\Gamma_{2}^{\star},v_{2})$ if and only if \newline $K_{1}\subseteq_{\mathfrak{L}_{ring}}K_{2}$, $\Gamma_{1}^{\star}\subseteq_{\mathfrak{L}_{oag}^{\star}}\Gamma_{2}^{\star}$ and $v_{2\restriction_{K_{1}}}=v_{1}$. Analogously, it holds for $(K_{1}^{\star},\Gamma_{1},v_{1})$ and $(K_{2}^{\star},\Gamma_{2},v_{2})$.\\
Thus, suppose that  $$(K_{1}^{\star},\Gamma_{1},v_{1})\nsubseteq_{\mathfrak{L}_{{vf}^{\star\star}_{\restriction_{K^{\star},\Gamma}}}}(K_{2}^{\star},\Gamma_{2},v_{2}).$$ \\ If $\Gamma_{1}\nsubseteq_{\mathfrak{L}_{oag}}\Gamma_{2}$, then clearly $\Gamma_{1}^{*}\nsubseteq_{\mathfrak{L}_{oag}^{*}}\Gamma_{2}^{*}$, and so $$(K_{1},\Gamma_{1}^{\star},v_{1})\nsubseteq_{\mathfrak{L}_{{vf}^{\star}_{\restriction_{K,\Gamma^{\star}}}}}(K_{2},\Gamma_{2}^{\star},v_{2}).$$ The same argument works if $v_{2}$ does not extend $v_{1}$. \\ Thus, suppose $K_{1}^{\star}\nsubseteq_{\mathfrak{L}_{ring}^{\star}}K_{2}^{\star}$. If $K_{1}^{\star}\nsubseteq K_{2}^{\star}$ as pure fields, we clearly have the assert. Thus, suppose $K_{1}^{\star}\subseteq K_{2}^{\star}$. We have three different cases \begin{itemize}
\item[i.] there are $x,y\in K_{1}^{\star}$ such that $x||^{n,m}y$ in $K_{1}^{\star}$ and $\neg\,x||^{n,m}y$ in $K_{2}^{\star}$. This holds if and only if $$v_{1}(x)|^{n,m}v_{1}(y) \,\,\text{in} \,\,\Gamma_{1}^{\star}\,\, \text{and}\,\, v_{2}(x)\not{|}^{n,m}v_{2}(y)\,\, \text{in}\,\, \Gamma_{2}^{\star},$$ and this is so if and only $\Gamma_{1}^{\star}\nsubseteq_{\mathfrak{L}^{\star}_{oag}}\Gamma_{2}^{\star}$, since $v_{2}$ extends $v_{1}$;
\item[ii.] $\$^{K_{1}^{\star}}\neq\$^{K_{2}^{\star}}\cap (K_{1}^{\star})^{2}$, thus there are $x,y\in K_{1}^{\star}$ such that $(x,y)\in(\$^{n,m})^{K_{1}^{\star}}$ and $(x,y)\notin(\$^{n,m})^{K_{2}^{\star}}$. This holds if and only if $$(v_{1}(x),v_{1}(y))\in(\overline{s}^{n,m})^{\Gamma_{1}^{\star}}\,\,\text{and}\,\, (v_{2}(x),v_{2}(y))\in(\overline{s}^{n,m})^{\Gamma_{2}^{\star}}$$ and since again $v_{2}$ extends $v_{1}$ we have that $(\overline{s}^{n,m})^{\Gamma_{1}^{\star}}\neq(\overline{s}^{n,m})^{\Gamma_{2}^{\star}}\cap \Gamma_{1}^{\star}.$ Thus $\Gamma_{1}^{\star}\nsubseteq_{\mathfrak{L}^{\star}_{oag}}\Gamma_{2}^{\star}$;
\item[iii.] $A^{K_{1}^{\star}}\neq A^{K_{2}^{\star}}\cup K_{1}^{\star}$, if and only if there is $x\in K_{1}^{\star}$ such that $$v_{1}(x)=^{\bullet}1_{\bullet}\,\,\text{and}\,\, v_{1}(x)=^{\bullet}1_{\bullet},$$ but since $v_{2}(x)=v_{1}(x)$ then $$(=^{\bullet}1_{\bullet})^{\Gamma_{1}^{\star}}\neq(=^{\bullet}1_{\bullet})^{\Gamma_{2}^{\star}}\cap\Gamma_{1}^{\star}.$$
\end{itemize}
\end{proof}

As for Corollary \ref{elemb}, the next result follows from Lemma \ref{bi-interpretability2} and Remark \ref{embedd2}.

\begin{cor} \label{elemb2} Let $\mathcal{K}_{1},\mathcal{K}_{2}$ be $\mathfrak{L}^{*}_{vf}$-structures such that $\mathcal{K}_{1}\subseteq\mathcal{K}_{2}$. Then $$\mathcal{K}_{1}\preceq_{\mathfrak{L}_{vf}^{\star}}\mathcal{K}_{2}\iff\widetilde{\mathcal{K}_{1}}\preceq_{\mathfrak{L}_{vf}^{\star\star}}\widetilde{\mathcal{K}_{2}}.$$\end{cor}
We conclude this section with the following remark. For practical purpose, we were interested in working with ordered abelian groups elementarily equivalent to the lexicographic sum of $\mathbb{Z}$ with a minimal positive element. At the beginning of Section \ref{mtlpz} we introduced it as an example of ordered abelian groups with spine $(\omega, <)$ and no colours. Thus, we have the following.
\begin{rem} Fact \ref{factlpz}, Proposition \ref{proplpz} and Theorem \ref{TheoremModelCompletenessProductZ} hold still if we replace $\bigoplus_{i<\omega^{*}}{\mathbb{Z}}$ with any group $G$ of order type $\omega$ and with no colours. \end{rem}

\section{Examples} \label{examples}
\begin{itemize}\item Consider the field $\mathbb{Q}_{p}((t^{\mathbb{Z}}))$ of Hahn series over $\mathbb{Q}_{p}$ in the indeterminate $t$ and value group $\mathbb{Z}$. The element $h=\sum_{e\in\mathbb{Z}}c_{e}t^{e}$ is a Hahn series if $supp(h)=\{e\in\mathbb{Z}\vert\,c_{e}\neq0\}$ is well ordered. The usual valuation $v_{t}$ over $\mathbb{Q}_{p}((t^{\mathbb{Z}}))$ of the Hahn construction is such that $v(0)=\infty$ and $v(h)=min\,supp(h)$ for every Hahn series $h\neq0$. Its residue field is $\mathbb{Q}_{p}$. Note that $v_{t}$ is a coarsening of a valuation $val$ defined over $\mathbb{Q}_{p}((t^{\mathbb{Z}}))$ with values in $\mathbb{Z}\times\mathbb{Z}$, that is a composition of the $p$-adic valuation $v_{p}$ over $\mathbb{Q}_{p}$ and the valuation $v_{t}$. Thus $$O_{val}=\left\{x\,\vert\,val(x)\geq0\right\}=\left\{x\,\vert\,v_{t}(x)>0\,\,\text{or}\,\,v_{t}(x)=0\wedge v_{p}(ac_{t}(x))\geq0\right\}.$$ Note that the ramification index is $e=1$, the residue field with respect to the valuation $val$ is $\mathbb{F}_{p}$ and the value group is an ordered abelian group with finite spines (with the lexicographic order), whose only non-trivial proper definable convex subgroup is $\{0\}+\mathbb{Z}$. \\ So, by Theorem \ref{TheoremModelCompletenessFiniteRamification}, $Th((\mathbb{Q}_{p}((t^{\mathbb{Z}})),val))$ is model complete in the language of rings together with two predicates $A_{0}, A_{1}$ such that $$\begin{array}{rl} A_{0}^{\mathbb{Q}_{p}((t^{\mathbb{Z}}))}&=\left\{x\,\vert\,val(x)\in(0,0)+(\{0\}+\mathbb{Z})\right\} \\
&=\left\{x\,\vert\,x=\sum_{i\geq0}a_{i}t^{i},a_{0}\neq0\right\}\end{array}$$ and $$\begin{array}{rl} A_{1}^{\mathbb{Q}_{p}((t^{\mathbb{Z}}))}&=\left\{x\,\vert\,val(x)\in(1,0)+(\{0\}+\mathbb{Z})\right\} \\
&=\left\{x\,\vert\,x=\sum_{i\geq1}a_{i}t^{i},a_{1}\neq0\right\}.\end{array}$$
\item The previous example can be generalized to the field of Hahn series over $\mathbb{Q}_{p}$ in $n$ indeterminates. We can consider, as before, the valuation $val_{n}$ over $\mathbb{Q}_{p}((t_{1}^{\mathbb{Z}}))\ldots((t_{n}^{\mathbb{Z}}))$, valued in $\bigoplus_{i=1}^{n+1}\mathbb{Z}$, such that $$O_{val_{n}}=\bigcup_{i=0}^{n}O^{i},$$ where $$\begin{array}{rl} O^{n}=&\{x\,\vert\,v_{t_{n}}(x)>0\} \\
O^{n-1}=&\{x\,\vert\,v_{t_{n}}(x)=0\wedge v_{t_{n-1}}(ac_{t_{n}}(x))>0\} \\
O^{n-2}=&\{x\,\vert\,v_{t_{n}}(x)=0\wedge v_{t_{n-1}}(ac_{t_{n}}(x))=0\wedge v_{t_{n-2}}(ac_{t_{n-1}}(ac_{t_{n}}(x))>0\}\\
&\vdots \\
O^{0}=&\{x\,\vert\,v_{t_{n}}(x)=0\wedge v_{t_{n-1}}(ac_{t_{n}}(x))=0\,\wedge v_{t_{n-2}}(ac_{t_{n-1}}(ac_{t_{n}}(x)))=0\,\wedge \\ 
&\ldots\wedge v_{t_{1}}(ac_{t_{2}}(\dots(ac_{t_{n}}(x))))=0\,\wedge v_{p}(ac_{t_{1}}(\dots(ac_{t_{n}}(x))..))>0\}\end{array}$$
Note that $val_{n}$ is unramified, the residue field is $\mathbb{F}_{p}$ and the value group is an ordered abelian group with finite spines (in particular it has finitely many convex definable subgroups). So, by Theorem \ref{TheoremModelCompletenessFiniteRamification}, the theory of the valued field $$\mathcal{K}=(\mathbb{Q}_{p}((t_{1}^{\mathbb{Z}}))\ldots((t_{n}^{\mathbb{Z}})),val_{n})$$ is model complete in the language of rings together with predicates $A_{i,0}, A_{i,1}$, for $i=0,\ldots,n$, such that 
$$\begin{array}{rl} A_{i,0}^{\mathcal{K}}&=\{x\in\mathcal{K}\,\vert\,val_{n}(x)\in H_{i}\} \\
&=\{x\in\mathcal{K}\,\vert\,x\in O^{i}\wedge v_{t_{i}}(ac_{t_{i+1}}(\ldots(ac_{t_{n}}(x))..))=0 \}, \end{array}$$
$$\begin{array}{rl} A_{i,1}^{\mathcal{K}}&=\{x\in\mathcal{K}\,\vert\,val_{n}(x)\in\,1_{i}+H_{i}\} \\
&=\{x\in\mathcal{K}\,\vert\,x\in O^{i}\wedge v_{t_{i}}(ac_{t_{i+1}}(\ldots(ac_{t_{n}}(x))..))=1 \}. \end{array}$$
\item The following valued field is an example for Theorem \ref{TheoremModelCompletenessProductZ}. Consider the field of Hahn series over $\mathbb{Q}_{p}$ in infinitely many indeterminates $$K=\bigcup_{n\in\mathbb{N}}\mathbb{Q}_{p}((t_{1}^{\mathbb{Z}}))\ldots((t_{n}^{\mathbb{Z}})).$$ We define a valuation $val_{\infty}$ over $K$ with values in $\bigoplus_{i<\omega^{*}}\mathbb{Z}$ from the valuations $val_{n}$ over $\mathbb{Q}_{p}((t_{1}^{\mathbb{Z}}))\ldots((t_{n}^{\mathbb{Z}}))$, such that $$O_{val_{\infty}}=\bigcup_{n\in\mathbb{N}}O_{val_{n}}.$$ By Theorem \ref{TheoremModelCompletenessProductZ}, we have a one-sorted language in which the valued field $\mathcal{K}=(K,val_{\infty})$ has a model complete theory.
\end{itemize}
 
In this paper, we focus on the cases in which there is a language for quantifier elimination for the theory of the value group. We would like to have an expansion of the language of the field structure which contains all the information on the structure of the value group. For an arbitrary oag, the best result on quantifier elimination is due to Cluckers and Halupczok in \cite{CH11}. Inspired by Schmitt's language (\cite{Schmitt82}), they provide a language with additional sorts for the spines of a group $G$ and prove quantifier elimination of $G$ relative to the spines. Thus, one natural problem is to identify a language where the finitely ramified valued field has relative quantifier elimination in the most general case in which the value group is an arbitrary ordered abelian group.

\vspace{3mm}


\textbf{Acknowledgements.} This work represents the first part of my Ph.D. thesis. I would like to thank my advisor Prof. Paola D'Aquino for her constant guidance and support. I would also like to thank Konstantinos Kartas for pointing out Remark \ref{perfect}, Martina Liccardo for the constant willingness to talk about mathematics, and Pierre Touchard for the many fruitful discussions.
\vspace{3mm}

\textbf{Conflicts of Interest.} None declared.
\vspace{3mm}

\bibliographystyle{acm}              
\bibliography{References2}

\end{document}